\documentclass{amsart}
\usepackage{amsmath,amssymb,txfonts,enumerate}
\usepackage{times}
\usepackage{bbm}
\usepackage[dvips]{graphicx}
\allowdisplaybreaks
\usepackage{amsxtra}
\usepackage[cp1252]{inputenc}
\usepackage{color}

\usepackage{mathrsfs}

\usepackage[colorlinks,citecolor=red,urlcolor=blue,bookmarks=false,hypertexnames=true,backref]{hyperref} 

\numberwithin{equation}{section}

\newtheorem{Theorem}{Theorem}[section]
\newtheorem{theorem}{Theorem}[section]

\newtheorem{remark}[theorem]{Remark}
\newtheorem{corollary}[theorem]{Corollary}

\newtheorem{defi}{\hskip\parindent {Definition}}[section]
\newtheorem{lemma}[theorem]{Lemma}

\newcommand{\ct}[1]{\langle {#1}\rangle \lower.3ex\mathrm{$_{t}$}}
\newcommand{\lt}[1]{[ {#1}] \lower.3ex\mathrm{$_{t}$}}

\newcommand{\BBB}{\mathcal{B}}

\newcommand{\SSS}{\mathcal{S}}

\numberwithin{equation}{section}

\begin{document}

	\title{Calder\'on's commutator on Stratified Lie groups}
	\author{Yanping Chen}
	\address{Yanping Chen, Department of Applied Mathematics, School of Mathematics and Physics,
		University of Science and Technology Beijing,
		Beijing 100083, China
	}
	\email{ yanpingch@126.com}
	
	\author{Zhenbing Gong}
	\address{\color{red}Zhenbing Gong, Department of Applied Mathematics, School of Mathematics and Physics,
		University of Science and Technology Beijing,
		Beijing 100083, China
	}
	\email{ gongzhenb@126.com }
	
	\author{Ji Li}
	\address{Ji Li, Department of Mathematics, Macquarie University, NSW 2109, Australia}
	\email{ji.li@mq.edu.au}
	
	\author{Edward McDonald}
	\address{Edward McDonald, Department of Mathematics, Penn State University, University Park, PA 16802, USA}
	\email{eam6282@psu.edu}
	
	\author{Dmitriy Zanin}
	\address{Dmitriy Zanin, School of Mathematics and Statistics, UNSW, Kensington, NSW 2052, Australia}
	\email{d.zanin@unsw.edu.au}

	\date{}
	
	\begin{abstract}
		Motivated by the recent work of Gimperlein and Goffeng on Calder\'on's commutator on compact Heisenberg type manifolds and the related weak Schatten class estimates, 
		we establish the characterisation of $L^p$ boundedness for Calderon's commutator on stratified Lie groups. We further study related weak Schatten class estimates for second order commutators on two step stratified Lie groups, which include the Heisenberg groups. This latter result is obtained using double operator integral techniques which are novel in this area.	
	\end{abstract}


	\begin{center}
	\end{center}

	\renewcommand{\thefootnote}{}
	
	

	\maketitle

	\arraycolsep=1pt
	
	\section{Introduction}
	
	Given a suitable function $b $, we denote by $M_b$ the multiplication operator defined by 
	$
	M_bf = b\cdot f
	$ for any measurable function $f.$ The commutator $[M_b, T ]$  formed by the singular integral operator $T$  and  $M_b$, which is defined by
	$[M_b, T ] f= M_{b}(T f) -T (M_{b} f )$,
	has played a vital role in harmonic analysis, complex analysis and partial differential equations. See for example \cite{C1965,CRW1976} .
	In 1965, Calder\'{o}n \cite{C1965} investigated the following operator on the real line $\mathbb{R}$ (now known as Calder\'{o}n's commutator)
	\begin{align*}
		\Big[M_b,\frac{d}{dx}H\Big]f(x):=\text{p.v.}\frac{1}{\pi}\int_{\mathbb{R}}\frac{b(x)-b(y)}{(x-y)^2}f(y)dy,
	\end{align*}
	and proved that $b\in {\rm Lip}(\mathbb{R})$ (the Lipschitz space on $\mathbb{R}$) is sufficient for the $ L^{p} (\mathbb{R})$ $(1<p<\infty)$ boundedness of  $[M_b,\frac{d}{dx}H]$. 
	
	Moreover, 
	Calder\'{o}n \cite{C1965} 
	proved that $[M_b,T_1]$ is bounded on $ L^{p}(\mathbb{R}^n)$  ($1 < p < \infty$) if $b\in {\rm Lip}(\mathbb R^n)$, where 
	\begin{align}	T_1f(x):=\text{p.v.}\int_{\mathbb{R}^n}\frac{\Omega(x-y)}{|x-y|^{n+1}}f(y)dy
	\end{align}
	with $\Omega\in L{\rm log}^{+}L(\mathbb S^{n-1})$ homogeneous of degree zero, and satisfying the vanishing moment condition of order one on the unit sphere $\mathbb S^{n-1}$:
	$$ \int_{\mathbb{S}^{n-1}}\Omega(x)x^{\alpha}d\sigma(x)=0$$
	for all multi-indices $\alpha$ of magnitude $|\alpha|=1$. Well-known examples 
	include $T_1 = \frac{\partial}{\partial \mathrm{X}_j}R_k$, where $R_k$ denotes the $k$-th Riesz transform on $\mathbb{R}^n$. 
	Calder\'{o}n's result has inspired various mathematicians to find different proofs, and to find further applications, see for example \cite{C1980,Co1981,CM1975,Ho1994,MD2017,CDH2016}.

	Recently, Gimperlein and Goffeng \cite{GG} studied Calder\'on's commutator for a class of pseudo differential operators of order one on certain compact Heisenberg manifolds, establishing the boundedness when the symbol $b$ belong to the corresponding Lipschitz space. Moreover, they studied the commutator of pseudo differential operators of order zero and obtained the weak Schatten class estimate at the critical index  when the symbol $b$ is in $C^\infty$. 
	
	Motivated by Gimperlein and Goffeng's work, 
	it is natural to explore  Calder\'on's commutators in a general setting, the stratified Lie groups. We aim to characterise the boundedness of Calder\'on's commutator $[M_b,\nabla R_j]$ on stratified nilpotent Lie group $\mathcal{G}$, where {$\nabla b \in  L^{\infty}(\mathcal{G})$  with $\nabla$ the gradient operator and $R_j$ is the $j$-th Riesz transform for $j=1,\ldots,n$. 
		We denote $\| b\|_{\operatorname{Lip}(\mathcal{G})}:=\|\nabla b\|_{L^{\infty}(\mathcal{G})} $. } The key step is to show that $[M_b,\nabla R_j]$ is a Calder\'on--Zygmund operator of order zero. Moreover, we explore the weak Schatten estimate of $[M_a,[M_b,\nabla R_j]]$ at the critical index for $a$ belonging to a corresponding homogeneous Sobolev space. We refer to some previous related work on $[M_b, R_j]$ on stratified Lie groups \cite{CDLW2019,DLLW2019, FLL, FLMSZ,MSZ_cwikel, LXY}.

	We now state our results explicitly.
	Suppose $\mathcal{G}$ is a stratified nilpotent Lie group of homogeneous dimension $\mathbb Q$, and let $\{\mathrm{X}_{j}\}_{1 \leq j \leq n}$ be a basis for the left-invariant vector fields of degree one on $\mathcal{G}.$ Let $\Delta=\sum_{j=1}^{n} \mathrm{X}_{j}^{2}$ be the sub-Laplacian on $\mathcal{G}.$ The $j$th Riesz transform on $\mathcal{G}$ is defined formally by
	$R_{j}:=\mathrm{X}_{j}(-\Delta)^{-\frac{1}{2}}.$ 
	We write $\nabla = (\mathrm{X}_1,\ldots,\mathrm{X}_n)^\top$ for the gradient operator.

	The first main result of this paper is as follows.
	\begin{Theorem}\label{Thom1}
Suppose $1<p<\infty$ and $j=1,2, \ldots, n$. If $b \in \operatorname{Lip}(\mathcal{G})$, then 
		$$
		\|[M_b,\nabla R_j]f\|_{L^{p}(\mathcal{G})}\leq c_{p,\mathcal{G}}\| b\|_{\operatorname{Lip}(\mathcal{G})}\|f\|_{L^{p}(\mathcal{G})}.
		$$
{On the other hand, if b is real valued and $[M_b,\nabla R_j]$ is bounded on $L^p(\mathcal G)$, then $b \in \operatorname{Lip}(\mathcal{G})$.}

	\end{Theorem}
{\begin{remark}
	The lower $L^2(\mathbb{R})$ bound for Theorem \ref{Thom1} for the Calder\'on's commutator on $\mathbb R$ is showed in \cite{Murray1985}.
\end{remark}}

	By noting that  $[M_b, \nabla]f=-(\nabla b)f$ whenever $b$ is differentiable, 
	and that by the Leibniz rule $$
	\nabla[M_b,R_{j}]=[M_b,\nabla R_{j}] -[M_b, \nabla]R_{j},$$
	it is easy to obtain the following result.
	\begin{corollary}
		Suppose $1<p<\infty$ and $j=1,2, \ldots, n$. If $b \in \operatorname{Lip}(\mathcal{G})$, then
		$$
		\|\nabla[M_b,  R_{j}] f\|_{L^{p}(\mathcal{G})}\leq c_{p,\mathcal G}\| b\|_{\operatorname{Lip}(\mathcal{G})}\|f\|_{L^{p}(\mathcal{G})}.
		$$
	\end{corollary}
	
	Similarly, using $(-\Delta)^{\frac12} = \sum_{j=1}^n -\mathrm{X}_jR_j,$ we deduce the following.
	\begin{corollary}\label{square root of delta corollary}
		If  $b \in \operatorname{Lip}(\mathcal{G})$, then the commutator $[(-\Delta)^{\frac12},M_b]$ is bounded on $L^p(\mathcal{G})$ for all $1<p<\infty$ with norm at most a constant multiple of $\|b\|_{\mathrm{Lip}(\mathcal{G})}.$
	\end{corollary}
	
	Next, we point out that our approach also works for certain higher order commutators. We state the result for the second-order commutator $[b,[b,\mathrm{X}_{\ell}R_{j}\mathrm{X}_{i}]]$.

	\begin{corollary}\label{cor2}
	{	Suppose  $1<p<\infty$ and $i,j,\ell=1,2, \ldots, n$. 
	If  $ b_1, b_2  \in \operatorname{Lip}(\mathcal{G})$, then 
	$$
	\|[M_{b_1},[M_{b_2},\mathrm{X}_{\ell}R_{j}\mathrm{X}_{i}]]\|_{L^{p}(\mathcal{G})\to L^{p}(\mathcal{G})}\leq c_{p,\mathcal G}\| b_1\|_{\operatorname{Lip}(\mathcal{G})}\| b_2\|_{\operatorname{Lip}(\mathcal{G})}.
	$$ 
	On the other hand, if $b_1$ and $b_2$ are real valued and if $[M_{b_1},[M_{b_2},\mathrm{X}_{\ell}R_{j}\mathrm{X}_{i}]]$ is bounded on $L^p(\mathcal G)$, then $b_1,b_2 \in \operatorname{Lip}(\mathcal{G})$.}
	\end{corollary}
	
	We also remark that if $b$ is in the closure of $C_0^\infty(\mathcal G)$ under the $\operatorname{Lip}(\mathcal G)$, then $[M_b,\nabla R_j]$, $j=1,2,\ldots,n$, is a compact Calder\'on--Zygmund operator. This can be deduced from the work of P. Villarroya \cite{V}.

		There has been some recent interest in the boundedness of higher order Calder\'on commutators, that is
		\[
		[M_{b_1},[M_{b_2},[,\ldots,[M_{b_n},\nabla R_j]]]]
		\]
		for appropriate functions $b_1,\ldots,b_n$ \cite{Lai}. In this theme, we consider the 		Schatten class of the commutator $[M_{b_{1}},[M_{b_2},\nabla R_j]]$ on $\mathcal G.$ 
	
	The degree of compactness can be measured by the scale of Schatten-von Neumann classes. Recall
	that if $T$ is a compact linear operator on a Hilbert space $H$, then the $(n+1)$th singular value of $T$ is the distance of $T$ from the space of rank $n$ operators. That is,
	\[
	\mu(n,T) = \inf\big\{\|T-R\|_{H\to H} \;:\;\mathrm{rank}(R)\leq n\big\}.
	\]
	Equivalently, $\mu(n,T)$ is the $(n+1)$th eigenvalue of the absolute value $|T|,$ arranged in non-increasing order with multiplicities. A compact operator $T$ on a Hilbert space is said to belong
	to $\SSS^{p,\infty}$ if
	\[
	\mu(n,T) = O((n+1)^{-\frac1p}),\quad n\to \infty,
	\]
	that is,
	\[
	\|T\|_{p,\infty} := \sup_{n\geq 0}\, (n+1)^{\frac{1}{p}}\mu(n,T) < \infty.
	\]
	Sufficient conditions for a second order commutator $[M_{b_1},[M_{b_2},\mathrm{X}_{k_1}R_{k_2}]]$ to belong to a Schatten-von Neumann ideal
	can be expressed using Sobolev spaces. The homogeneous Sobolev space  $\dot{W}^{k,p}(\mathcal{G})$ is defined for $1\leq p\leq \infty$ and $k=1,2,3,\ldots$ as the space of all distributions $u$ on $\mathcal{G}$ such that
	\[
	\mathrm{X}_{i_1}\cdots \mathrm{X}_{i_j}u \in L^p(\mathcal{G}),\quad 1\leq j\leq k
	\]
	for all $i_1,\ldots,i_k \in\{ 1,\ldots,n\}$ with corresponding semi-norm
	\[
	\|u\|_{\dot{W}^{k,p}(\mathcal{G})} := \sum_{j=1}^k\sum_{1\leq i_1,\ldots,i_j\leq n} \|\mathrm{X}_{i_1}\cdots \mathrm{X}_{i_n}u\|_{L^p(\mathcal{G})}.
	\]
	For further details see \cite[Section 4]{Folland1975}.

	\begin{Theorem}\label{main3}
		Suppose that $\mathcal{G}$ is a two-step stratified Lie group with homogeneous dimension $\mathbb Q>2$ as above, and and $k_1,k_2=1,2, \ldots, n.$ If
		$$
		b_1 \in \dot{W}^{1,\mathbb Q}(\mathcal{G}),\quad b_2 \in \dot{W}^{1,\infty}(\mathcal{G})\cap \dot{W}^{2,\mathbb Q}(\mathcal{G}),
		$$
		then the double commutator $[M_{b_1},[M_{b_2}, \mathrm{X}_{k_1} R_{k_2}]]$ belongs to the weak Schatten class $\SSS^{\mathbb Q,\infty}.$
		Moreover, we have the estimate
		\[
		\|[M_{b_1},[M_{b_2},\mathrm{X}_{k_1}R_{k_2}]]\|_{\mathbb Q,\infty} \lesssim \|b_1\|_{\dot{W}^{1,\mathbb Q}}(\|b_2\|_{\dot{W}^{2,\mathbb Q}(\mathcal{G})}+\|b_2\|_{\dot{W}^{1,\infty}(\mathcal{G})}),
		\]
		where the implicit constant is independent of $b_1$ and $b_2.$
	\end{Theorem}
	In other words, Theorem \ref{main3} asserts the $n$th largest singular value of $[M_{b_1},[M_{b_2},\mathrm{X}_{k_1}R_{k_2}]]$ is $O((n+1)^{-\frac{1}{\mathbb{Q}}}).$

	Theorem \ref{main3} is partly inspired by a result of Gimperlein and Goffeng \cite[Theorem 1.3]{GG}, which asserts (in particular) that the commutator of a $0$-order Heisenberg pseudodifferential operator on a compact Heisenberg manifold with a Lipschitz function belongs to a certain weak Schatten class. Our result is different insofar as the underlying manifold is not compact and so we need some integrability properties on the functions $b_1$ and $b_2.$
	
	This paper is organised as follows. In Section \ref{Section P}, we provides preliminary
	background information and notation on
	stratified nilpotent Lie groups. In Section \ref{Section T1}, the proof of Theorem \ref{Thom1} is given. In Section
	\ref{Section co}, we give the proof of Corollary \ref{cor2}. In the last section  we prove Theorem \ref{main3}.
	
	\section{Preliminaries}\label{Section P}
	
	\subsection{The stratified nilpotent Lie groups $\mathcal{G}$}
	
	Let $\mathfrak{g}$ be the left-invariant Lie algebra (assumed real and of finite dimension) of the connected, simply connected nilpotent Lie group $\mathcal{G}$. A Lie group $\mathcal{G}$ of is called stratified of step $k$ if $\mathfrak{g}$  satisfies a direct sum decomposition
	$$
	\mathfrak{g}=\bigoplus_{i=1}^{k} V_{i}, \quad\left[V_{1}, V_{i}\right]=V_{i+1}, \text { for } i \leq k-1 \text { and }\left[V_{1}, V_{k}\right]=0.
	$$ 
	The natural dilations on $\mathfrak{g}$ are given by
	$$
	\delta_{r}\bigg(\sum_{i=1}^{k} v_{i}\bigg)=\sum_{i=1}^{k} r^{i} v_{i}, \quad \text { with } v_{i} \in V_{i}, ~ r>0.
	$$ The dilations on $\mathfrak{g}$ allow the definition of dilation on $\mathcal{G}$, which we still denote by $\delta_{r}$.
	
	One identifies $\mathfrak{g}$ and $\mathcal{G}$ via the exponential map
	$$
	\exp : \mathfrak{g} \longrightarrow \mathcal{G}
	$$
	which is a diffeomorphism \cite[Proposition 1.2a]{FS1982}. We write the group operation on $\mathcal{G}$ as $\circ,$ and the identity element as $o.$
	
	We fix once and for all a (bi-invariant) Haar measure $d g$ on $\mathcal{G}$ (which is just the lift of Lebesgue measure on $\mathfrak{g}$ via exp) \cite[Proposition 1.2c]{FS1982}.

	We choose once and for all a basis $\left\{\mathrm{X}_{1}, \cdots, \mathrm{X}_{n}\right\}$ of $V_{1}$ and consider the sub-Laplacian $\Delta=\sum_{j=1}^{n} \mathrm{X}_{j}^{2}$. Observe that $\mathrm{X}_{j}(1 \leq j \leq n)$ is homogeneous of degree 1 with respect to the dilations, and $\Delta$ has degree 2 in the sense that
	$$
	\begin{aligned}
		&\mathrm{X}_{j}\left(f \circ \delta_{r}\right)=r\left(\mathrm{X}_{j} f\right) \circ \delta_{r}, \quad 1 \leq j \leq n, r>0, f \in C^{1} \\
		&\delta_{\frac{1}{r}} \circ \Delta \circ \delta_{r}=r^{2} \Delta, \quad \forall r>0.
	\end{aligned}
	$$
	Because of the stratification, the dilation $\delta_{\lambda}$ defines an automorphism of $\mathcal{G}.$ A second important family of diffeomorphisms of $\mathcal{G}$ are the translations. For any $g \in \mathcal{G}$, the (left) translation $\tau_{g}: \mathcal{G} \rightarrow \mathcal{G}$ is defined as
	$$
	\tau_{g}\left(g^{\prime}\right)=g \circ g^{\prime}.
	$$
	For any set $E \subset \mathcal{G}$, denote by $\tau_{g}(E)=\left\{g \circ g^{\prime}: g^{\prime} \in E\right\}$ and $\delta_{r}(E)=\left\{\delta_{r}(g): g \in E\right\}$.
	
	For $i=1, \cdots, k$, let $n_{i}=\operatorname{dim} V_{i},$ $m_{i}=n_{1}+\cdots+n_{i}$ and $m_{0}=0$, clearly, $n_{1}=n$. Set $N=m_{k}$. T
	Identifying $V_i$ with $\mathbb{R}^{n_i}$ and $\mathcal{G}$ with $\mathbb{R}^N,$ we have
	$$
	\delta_{\lambda}(g)=\delta_{\lambda}\left(g_{1}, g_{2}, \cdots, g_{N}\right)=\left(\lambda^{\alpha_{1}} g_{1}, \lambda^{\alpha_{2}} g_{2}, \cdots, \lambda^{\alpha_{N}} g_{N}\right),\quad g_1,\ldots,g_N \in \mathbb{R}.
	$$
	where $\alpha_{j}=i$ whenever $m_{i-1}<j \leq m_{i}, i=1, \cdots, k$. Therefore, $1=\alpha_{1}=\cdots=\alpha_{n_{1}}<$ $\alpha_{n_{1}+1}=2 \leq \cdots \leq \alpha_{N}=k$.

	Let $\mathbb Q$ denote the homogeneous dimension of $\mathcal{G}$, namely,
	$$
	\mathbb Q=\sum_{i=1}^{N}\alpha_{i}=\sum_{i=1}^{k} i \operatorname{dim} V_{i}.$$

In this paper, $\rho(g,g')$ denotes the Carnot--Caratheodory distance between $g$ and $g'.$ For brevity, we denote $\rho(g,1_{\mathcal{G}})$ by $\rho(g).$

We define the ball centered at $g$ with radius $r$ by $B(g, r) = \{g' \in \mathcal{G} : \rho(g,g') < r\}$, and by  $\lambda B$  denote the ball $ B(g,\lambda r)$ with $\lambda>0$. Let $|B(g,r)|$ be the Haar measure of the ball $ 	B(g, r)
.$ Since the Haar measure is precisely the Lebesgue measure, it follows that 	$|B(g,r)|=r^{\mathbb{Q}}$ for every $g \in\mathcal{G}$ and for every $r>0.$

	\subsection{Lipschitz spaces ${\rm Lip}(\mathcal G)$ on $\mathcal{G}$}
We recall that $b\in {\rm Lip}(\mathcal G)$ if $\nabla b\in L^\infty (\mathcal G)$, and 
$\|b\|_{ {\rm Lip}(\mathcal G)}:=\| \nabla b\|_{ L^\infty (\mathcal G)}. $

\begin{theorem}\label{lipschitz theorem} If $\nabla b\in L_{\infty}(\mathcal{G}),$ then
$$|b(g_1)-b(g_2)|\leq c_{\mathcal{G}}\rho(g_1,g_2).$$ 
\end{theorem}
\begin{proof}
It is direct to see that via mean value theorem (\cite[Theorem 1.41]{FS1982}), 
$$ |b(g_1)-b(g_2)|\leq C \| \nabla b\|_{ L^\infty (\mathcal G)}\rho(g_1,g_2).$$
We denote the smallest constant $C$ by $C_0$ and set $c_{\mathcal G}:=C \| \nabla b\|_{ L^\infty (\mathcal G)}$.
Hence, 
$$ \sup_{ g_1\not=g_2 } { |b(g_1)-b(g_2)|\over \rho(g_1,g_2) }\leq c_{\mathcal G}.$$
The proof is complete.
\end{proof}
	
	Mac\'{i}as and Segovia \cite{MS1979} prove the following useful characterization  on general spaces of homogeneous type in the sense of Coifman and Weiss, which can be applied to our setting $\mathcal G$.
	\begin{lemma}[\cite{MS1979}]\label{eq-Rlem1} 
	Let $b$ be a locally integrable function on $\mathcal G$. Then the following two conditions are equivalent.
	
	{\rm(1)}. there exists $c_{\mathcal G}>0$ such that 
	$$ \sup_{ g_1\not=g_2 } { |b(g_1)-b(g_2)|\over \rho(g_1,g_2) }\leq c_{\mathcal G}.$$

{\rm(2)}. there exists $C_{\mathcal G}>0$ such that 
		$$\sup_B\frac{1}{|B|^{1+1 / \mathbb Q}} \int_{B}|b(g)-m_Bb| d g \leq C_{\mathcal G},$$
			where $m_Bb=\frac{1}{|B|} \int_Bb(g)dg,$ and $\sup_{B}$ denotes the supremum over all balls (with respect to Carnot--Caratheodory distance) in $\mathcal{G}.$ 
	\end{lemma}
	
\section{The proof of Theorem \ref{Thom1}}\label{Section T1}
\subsection{The upper bound on the norm of the commutator $[M_b, \nabla R_{j}]$}
	To prove the upper bound of $L^p\to L^p$ norm of the commutator $[M_b, \nabla R_{j}]$, we need to use the following results from the Calder\'on--Zygmund theory on $\mathcal{G}.$
	
	The following is a standard kernel definition, as in, e.g., \cite[Definition 1.8]{DH2009}.
	
	\noindent
\begin{defi}\label{def CZO} 
We say that a measurable function $K: \mathcal{G} \times \mathcal{G} \backslash\{g=g'\} \rightarrow \mathbb{C}$ is a Calder\'on--Zygmund kernel (or simply a \emph{standard kernel}) if 
		\begin{align*}
			&|K(g, g')| \leq c_K\frac{1}{\rho(g,g')^{\mathbb Q}},\\
			&|K(g, g')-K(g, g'')| \leq c_K \frac{\rho(g',g'')}{\rho(g,g')^{\mathbb Q+1}} \quad \text{if} \quad \rho(g,g')>2\rho(g',g''),
			\\
			&|K(g, g')-K(\tilde{g}, g')| \leq c_K \frac{\rho(g,\tilde{g})}{\rho(g,g')^{\mathbb Q+1}} \quad \text{if} \quad \rho(g,g')>2\rho(g,\tilde{g}).
		\end{align*}		
A linear operator $T:L^1_{\mathrm{comp}}(\mathcal{G})\to L^1_{\mathrm{loc}}(\mathcal{G})$ is said to be a Calder\'on--Zygmund singular integral operator if there exists a standard kernel $K$ such that for all compactly supported $L^1$ functions $f$ we have
		$$
		T f(g)=\int_{\mathcal{G}} K(g, g') f(g') d g', \quad g \notin \operatorname{supp}(f).
		$$
	\end{defi}
	
	The following $T1$ theorem for $\mathcal{G}$ is a special case of the corresponding theorem for spaces of homogeneous type, see \cite[Theorem 1.18]{DH2009}.
	\begin{theorem}\label{TTstar}
		Let $T$ be a Calder\'{o}n--Zygmund singular integral operator. Then a necessary and sufficient condition for the extension of $T$ as a continuous linear operator on $L^{2}(\mathcal{G})$  is that the following properties are all satisfied
		\begin{enumerate}[{\rm (1)}]
			
			\item $T 1 \in \operatorname{BMO}(\mathcal G)$,
			
			\item $T^{*} 1 \in \operatorname{BMO}(\mathcal G)$,
			
			\item $T$ satisfies the Weak Boundedness Property (WBP), i.e. 
			$$
			|\langle T \varphi, \psi\rangle| \leq c_T R^{\mathbb Q}\{\|\varphi\|_{L^\infty}+R\| \varphi\|_{{\rm Lip}(\mathcal{G})}\}\{\|\psi\|_{L^\infty}+R\| \psi\|_{{\rm Lip}(\mathcal{G})}\}
			$$
			for all $R>0$ and $g \in \mathcal{G}$, and all Lipschitz functions $\varphi, \psi$ supported in the ball
			$$
			B(g, R):=\{g' \in \mathcal{G}:\rho(g,g')<R\}.
			$$
		\end{enumerate}
	\end{theorem}
	
	The classical Calder\'on--Zygmund theorem allows us to deduce $L^p$ boundedness from $L^2$ boundedness. This follows from standard methods, see e.g. \cite[Theorem 2.1, Chapter III]{CoifmanWeiss1971}.
	\begin{lemma}\label{Rlemma2.2}
		Let $T$ be a Calder\'on--Zygmund singular integral operator which is a bounded operator on $L^{2}(\mathcal{G}).$
		Then $T$ is bounded on $L^p(\mathcal G), 1<p<\infty$.
	\end{lemma}
	The combination of Theorem \ref{TTstar} and Lemma \ref{Rlemma2.2} shows that if $T$ is a Calder\'on--Zygmund singular integral operator, $T1, T^*1 \in \mathrm{BMO}(\mathcal{G})$ and $T$ has the weak boundedness property, then $T$ is bounded on $L^p(\mathcal{G})$ for all $1<p<\infty.$ This forms the basis for our proof of Theorem \ref{Thom1}. 

\begin{lemma}\label{first cz verification lemma}
If $T$ is a convolution operator with the convolution kernel $K$ satisfying 
\begin{align}
&|K(g)| \leq c_K \frac{1}{\rho(g)^{\mathbb Q+1}},\label{eeee1}\\
&|K(g')-K(g'')| \leq c_K\frac{\rho(g',g'')}{\rho(g')^{\mathbb Q+2}} \quad \text{if} \quad \rho(g')>2\rho(g',g''),\label{eeee2},
\end{align}
then the kernel of $[M_b,T]$ is a standard kernel as in Definition \ref{def CZO}.
\end{lemma}
\begin{proof}
Denote by $\mathcal{K}$ the integral kernel of $[M_b,T].$ Note that
\begin{align*}
\mathcal{K}(g,g')=(b(g)-b(g'))K({g'}^{-1}g). 
\end{align*}
Appealing to the definition of $\mathcal{K},$ assumption \eqref{eeee1} and Theorem \ref{lipschitz theorem}, we obtain
$$|\mathcal{K}(g,g')|\leq |b(g)-b(g')|\cdot |K({g'}^{-1}g)|\leq c_{\mathcal{G}}\|b\|_{\rm Lip(\mathcal{G})}\rho(g,g')\cdot  c_K\frac{1}{\rho(g,g')^{\mathbb{Q}+1}}.$$ 
This verifies the first condition in Definition \ref{def CZO}.

Appealing to the definition of $\mathcal{K}$ and Theorem \ref{lipschitz theorem}, we obtain
\begin{align*}
| \mathcal{K}(g,g') -\mathcal{K}(g,g'') |&=|(b(g)-b(g'))K({g'}^{-1}g)-(b(g)-b(g''))K({g''}^{-1}g)|\nonumber\\
&\leq |b(g'')-b(g')|\cdot |K({g''}^{-1}g)|+|b(g)-b(g')|\cdot|K({g'}^{-1}g)-K({g''}^{-1}g)|\\
&\leq c_{\mathcal{G}}\|b\|_{{\rm Lip}(\mathcal{G})}\Big(\rho(g',g'')\cdot |K({g''}^{-1}g)|+\rho(g,g')\cdot |K({g'}^{-1}g)-K({g''}^{-1}g)|\Big).\nonumber
\end{align*}
Suppose now $\rho(g,g')>2\rho(g',g'').$ It follows that $\rho({g'}^{-1}g)>2\rho({g'}^{-1}g,{g''}^{-1}g).$ Appealing to \eqref{eeee1} and \eqref{eeee2}, we estimate
\begin{align*}
	| \mathcal K(g,g') -\mathcal K(g,g'') |
&\leq c_{\mathcal{G}}c_K\|b\|_{{\rm Lip}(\mathcal{G})}\Big(\rho(g',g'')\cdot\frac1{\rho(g,g'')^{\mathbb{Q}+1}}+\rho(g,g')\cdot \frac{\rho(g',g'')}{\rho(g,g')^{\mathbb{Q}+2}}\Big)\nonumber\\
&\leq c_{\mathcal{G}}c_K\|b\|_{{\rm Lip}(\mathcal{G})}\Big(\rho(g',g'')\cdot\frac{2^{\mathbb{Q}+1}}{\rho(g,g')^{\mathbb{Q}+1}}+\rho(g,g')\cdot \frac{\rho(g',g'')}{\rho(g,g')^{\mathbb{Q}+2}}\Big).
\end{align*}
This verifies the second condition in Definition \ref{def CZO}. Third condition can be verified similarly.
\end{proof}

\begin{lemma}\label{first czo lemma}	
The convolution kernel $K^i_j$ of the operator $X_iR_j$ satisfies the conditions \eqref{eeee1}--\eqref{eeee2}. 
\end{lemma}	
\begin{proof} Let $p_t,$ $t>0,$ be the convolution kernel of $e^{t\Delta}$ on $\mathcal{G}.$. For convenience, we denote $p(g)=p_1(g).$ It is easy to see by a rescaling that
\begin{align}\label{pt}
p_{t}(g)=t^{-\frac{\mathbb Q}{2}} p(\delta_{\frac{1}{\sqrt{t}}}(g)), \quad \forall t>0,\quad g \in \mathcal{G}.
\end{align}
The convolution kernel of $X_iX_j(-\Delta)^{-\frac{1}{2}}$ is written as
$$K^i_j(g)=\frac{1}{\sqrt{\pi}} \int_0^{+\infty} t^{-\frac{1}{2}} (X_iX_jp_t)(g)d t=\frac{1}{\sqrt{\pi}} \int_{0}^{+\infty} t^{-\frac{\mathbb{Q}}{2}-\frac{3}{2}} (X_iX_jp)(\delta_{\frac{1}{\sqrt{t}}}(g))dt.$$
	
Classical estimates for the heat kernel and its derivatives on stratified groups (see Theorem 4.2 in \cite{VSC1992}) imply that
\begin{align}\label{kernel}
K^i_j(g)=O(\rho(g)^{-\mathbb{Q}-1}),\quad (X_mK^i_j)(g)=O(\rho(g)^{-\mathbb{Q}-2}),\quad g\in\mathcal{G}.
\end{align}
This immediately yields the assertion.
\end{proof}
	
	We further show that $T 1 \in \operatorname{BMO}(\mathcal G)$,  $T^{*} 1 \in \operatorname{BMO}(\mathcal G)$, and $T$ satisfies the weak boundedness property.

\begin{lemma}\label{lem X_i R_j( 1 ) =0}
For $i,j =1,2,\ldots,n$, we have $\mathrm{X}_i R_j( 1 ) =0$.
\end{lemma}
\begin{proof}
Recall that $R_j = \mathrm{X}_j \Delta^{-{1\over2}}$ is a standard Calder\'on--Zygmund operator on 
$\mathcal G$. Hence,  $R_j$ is bounded from $L^\infty(\mathcal G)$ to $\operatorname{BMO}(\mathcal G)$ (see for example the argument in \cite[Section 6]{FS1982}).

Thus, $R_j( 1 )$ is a function in $\operatorname{BMO}(\mathcal G)$ with 
$\|R_j( 1 )\|_{\operatorname{BMO}(\mathcal G)}\lesssim1$. Moreover, since
\begin{align*}
R_j( 1 ) = R_j(\delta_r \circ1 ) =\delta_r \circ R_j(1 ) 
\end{align*}
for all $r>0$, we obtain that 
$R_j( 1 )  $ must be a constant.

As a consequence,  $\mathrm{X}_i R_j( 1 ) =0$.

The proof is complete.
\end{proof}

\begin{lemma} \label{lem ijk}
For $i,j,k =1,2,\ldots,n$, 
$\mathrm{X}_i\mathrm{X}_j \Delta^{-1} R_k^*$ is bounded from $L^\infty(\mathcal G) $ to $\operatorname{BMO}(\mathcal G)$.
\end{lemma}
\begin{proof}
Note that $$ \mathrm{X}_i\mathrm{X}_j \Delta^{-1} R_k^* = \mathrm{X}_i \mathrm{X}_j \Delta^{-{3\over2}} X_k, $$
which is a standard Calder\'on--Zygmund operator on $\mathcal G$. Hence, it maps 
$L^\infty(\mathcal G) $ to $\operatorname{BMO}(\mathcal G)$ (see for example the argument in \cite[Section 6]{FS1982}).
\end{proof}

\begin{lemma}\label{lem BMO argument}
 For $i,j =1,2,\ldots,n$,  and for every $b\in {\rm Lip}(\mathcal{G})$, $\mathrm{X}_iR_jb\in \operatorname{BMO}(\mathcal G)$.
\end{lemma}
\begin{proof}
We note that 
\begin{align*}
\mathrm{X}_iR_jb &=  \mathrm{X}_i\mathrm{X}_j \Delta^{-1} \Delta^{{1\over2}} b 
=\mathrm{X}_i\mathrm{X}_j \Delta^{-1} \bigg(\sum_{k=1}^n R_k^* X_k b\bigg)
=  \sum_{k=1}^n \bigg( \mathrm{X}_i\mathrm{X}_j \Delta^{-1}R_k^*\bigg)( X_k b). 
\end{align*}

Note that $X_kb\in L^\infty (\mathcal G)$ for all $k=1,\ldots,n$, and that $\mathrm{X}_i\mathrm{X}_j \Delta^{-1} R_k^*$ is bounded from $L^\infty(\mathcal G)$ to $\operatorname{BMO}(\mathcal G)$ (by Lemma \ref{lem ijk}). We obtain that $\mathrm{X}_iR_jb\in \operatorname{BMO}(\mathcal G)$ with 
$$ \|\mathrm{X}_iR_jb\|_{ \operatorname{BMO}(\mathcal G)}\lesssim \sum_{j=1}^n \|\mathrm{X}_jb\|_{ L^\infty (\mathcal G) }.$$

The proof is complete. 
\end{proof}

\begin{lemma}\label{lem T1 T*1}
If $T=[M_b, \mathrm{X}_i R_{j}],$ then $T(1),T^{\ast}(1) \in \operatorname{BMO}(\mathcal G)$.
\end{lemma}
\begin{proof} Without loss of generality, $b$ is real-valued.

For $T(1)$, note that
$$T(1)= [M_b, \mathrm{X}_iR_j ](1)= b\cdot \mathrm{X}_iR_j(1)  -(\mathrm{X}_jR_i)(b)= -(\mathrm{X}_jR_i)(b),$$
where the last equality follows from Lemma \ref{lem X_i R_j( 1 ) =0}.
Moreover, from Lemma \ref{lem BMO argument} we see that $\mathrm{X}_jR_i(b) \in \operatorname{BMO}(\mathcal G)$.
Hence, we obtain that $T(1)\in \operatorname{BMO}(\mathcal G)$.

For $T^{\ast}(1)$, note that
$$	T^*(1)  =- [ b, (\mathrm{X}_jR_i)^* ](1)  
	= -b\cdot (\mathrm{X}_jR_i)^*(1)  +(\mathrm{X}_jR_i)^*(b)    
         = -b\cdot R_i^* \mathrm{X}_j(1)   + R_i^* \mathrm{X}_j (b)=R_i^{\ast}(\mathrm{X}_jb),   
$$  
where the last equality follows from $\mathrm{X}_j(1)=0.$
Note that  $\mathrm{X}_j(b) \in L^\infty(\mathcal G)$.  Again,  $R_i^*$  maps $L^\infty(\mathcal G)$ to $\operatorname{BMO}(\mathcal G)$ (see for example the argument in \cite[Section 6]{FS1982}).
Hence, 
	\begin{align*}
	\|T^*(1)\|_{\operatorname{BMO}(\mathcal G)}  \leq \|R_i^* \mathrm{X}_j (b)  \|_{\operatorname{BMO}(\mathcal G)}\leq C    \|\mathrm{X}_j(b) \|_{ L^\infty(\mathcal G)}.
         \end{align*}
Hence,
$T^*(1)  \in \operatorname{BMO}(\mathcal G)$.

The proof is complete.
\end{proof}

\begin{lemma}\label{pre-wbp first lemma} If $b\in {\rm Lip}(\mathcal{G}),$ then
$$\|M_{\chi_{B(g,R)}}[M_b,R_j]M_{\chi_{B(g,R)}}\|_{L^2(\mathcal{G})\to L^2(\mathcal{G})}=O(R).$$
\end{lemma}
\begin{proof} Integral operator with an integral kernel
$$(g_1,g_2)\to \rho(g_1,g_2)^{1-\mathbb{Q}}\cdot\chi_{B(1_{\mathcal{G}},1)}(g_1)\cdot\chi_{B(1_{\mathcal{G}},1)}(g_2)$$
is bounded. By dilation and translation invariance, integral operator with an integral kernel
$$(g_1,g_2)\to \rho(g_1,g_2)^{1-\mathbb{Q}}\cdot\chi_{B(g,R)}(g_1)\cdot\chi_{B(g,R)}(g_2)$$
is also bounded and its norm is proportional to $R.$

If $T_1$ and $T_2$ are integral operators with integral kernels $K_1$ and $K_2$ such that $|K_2|\leq K_1,$ then
$$\|T_2\|_{L^2(\mathcal{G})\to L^2(\mathcal{G})}\leq\|T_1\|_{L^2(\mathcal{G})\to L^2(\mathcal{G})}.$$

It remains to note that the integral kernel of $M_{\chi_{B(g,R)}}[M_b,R_j]M_{\chi_{B(g,R)}}$ is given by the formula
$$(g_1,g_2)\to (b(g_1)-b(g_2))\cdot K^j(g_2^{-1}g_1)\cdot \chi_{B(g,R)}(g_1)\cdot\chi_{B(g,R)}(g_2)=O\Big(\rho(g_1,g_2)^{1-\mathbb{Q}}\cdot \chi_{B(g,R)}(g_1)\cdot\chi_{B(g,R)}(g_2)\Big).$$
The proof is complete.
\end{proof}

\begin{lemma}\label{lem for wbp}
The operator $T=[M_b, \mathrm{X}_i R_{j}]$ satisfies the Weak Boundedness Property.
\end{lemma}

\begin{proof} 
	
Choose an arbitrary $g\in\mathcal G$ and $R>0$. Let $\varphi, \psi$ be  Lipschitz functions supported in the ball
	$$
	B({g}, R):=\{g'\in \mathcal{G}:\rho({g},g')<R\}.
	$$

Clearly,
$$T=-M_{\mathrm{X}_ib} R_j+\mathrm{X}_i[M_b,R_j].$$
Thus,
$$\langle T\phi,\psi\rangle=-\langle R_j(\phi),\mathrm{X}_i(b)\cdot\psi\rangle-\langle [M_b,R_j]\phi,\mathrm{X}_i\psi\rangle=-\langle R_j(\phi),\mathrm{X}_i(b)\cdot\psi\rangle-\langle M_{\chi_{B(g,R)}}[M_b,R_j]M_{\chi_{B(g,R)}}\phi,\mathrm{X}_i\psi\rangle.$$
and
\begin{align*}
|\langle T\phi,\psi\rangle|&\leq |\langle R_j(\phi),\mathrm{X}_i(b)\cdot\psi\rangle|+|\langle M_{\chi_{B(g,R)}}[M_b,R_j]M_{\chi_{B(g,R)}}\phi,\mathrm{X}_i\psi\rangle|\\
&\leq\|R_j\|_{L^2(\mathcal{G})\to L^2(\mathcal{G})}\|\phi\|_{L^2(\mathcal{G})}\|X_i(b)\cdot\psi\|_{L^2(\mathcal{G})}+\|M_{\chi_{B(g,R)}}[M_b,R_j]M_{\chi_{B(g,R)}}\|_{L^2(\mathcal{G})\to L^2(\mathcal{G})}\|\phi\|_{L^2(\mathcal{G})}\|X_i\psi\|_{L^2(\mathcal{G})} \\
&\leq R^{\mathbb{Q}}\Big(\|\phi\|_{L^{\infty}(\mathcal{G})}\|X_i(b)\|_{L^{\infty}(\mathcal{G})}\|\psi\|_{L^{\infty}(\mathcal{G})}+\|M_{\chi_{B(g,R)}}[M_b,R_j]M_{\chi_{B(g,R)}}\|_{L^2(\mathcal{G})\to L^2(\mathcal{G})}\|\phi\|_{L^{\infty}(\mathcal{G})}\|X_i\psi\|_{L^{\infty}(\mathcal{G})}\Big).
\end{align*}
The assertion follows now from Lemma \ref{pre-wbp first lemma}.
\end{proof}

\begin{proof}[Proof of the upper bound in Theorem \ref{Thom1}] Lemmas \ref{first czo lemma}, \ref{lem T1 T*1} and \ref{lem for wbp} assert that the assumptions in Theorem \ref{TTstar} are met for the operator $T=[M_b,\nabla R_j].$ By Theorem \ref{TTstar}, $T$ is bounded on $L^2(\mathcal{G}).$ Using Lemma \ref{Rlemma2.2}, $T$ is bounded on $L^p(\mathcal{G})$ for every $1<p<\infty.$ Using Leibniz rule, we infer the boundedness of the operator $\mathrm{X}_i[M_b,R_j]$ on $L^p(\mathcal{G})$ for every $1<p<\infty.$

What remains is to prove that the norm of $\mathrm{X}_i[M_b,R_j]$ is bounded by $\|b\|_{{\rm Lip}(\mathcal{G})}.$ To prove this, we use the Closed Graph Theorem. Let $X=\{b\in{\rm Lip}(\mathcal{G}):\ b(1_{\mathcal{G}})=0\}.$ Clearly, $(X,\|\cdot\|_{{\rm Lip}(\mathcal{G})})$ is a Banach space. We claim that the graph of the mapping $A:b\to \mathrm{X}_i[M_b,R_j]$ from $X$ to $B(L^p(\mathcal{G}))$ is closed.

Indeed, let $b_k\to b$ in $X$ and be such that $A(b_k)\to A$ in $B(L^p(\mathcal{G})).$ Let $\phi,\psi\in C^{\infty}_c(\mathcal{G}).$ We have
$$\langle A(b_k)\phi,\psi\rangle=-\langle[M_{b_k},R_j]\phi,\mathrm{X}_i\psi\rangle=\langle R_j(b_k\phi),\mathrm{X}_i\psi\rangle-\langle b_k\cdot R_j\phi,\mathrm{X}_i\psi\rangle.$$
We have
$$|(b_k-b)(g)|=|(b_k-b)(g)-(b_k-b)(1_{\mathcal{G}})|\leq c_G\|b_k-b\|_{{\rm Lip}(\mathcal{G})}\cdot\rho(g).$$
Hence, $b_k\to b$ pointwise. Hence, $b_k\cdot R_j\phi\to b\cdot R_j\phi$ and $R_j(b_k\phi)\to R_j(b\phi)$ in $L^2(\mathcal{G}).$ In particular,
$$\langle R_j(b_k\phi),\mathrm{X}_i\psi\rangle\to \langle R_j(b\phi),\mathrm{X}_i\psi\rangle,\quad \langle b_k\cdot R_j\phi,\mathrm{X}_i\psi\rangle\to\langle b\cdot R_j\phi,\mathrm{X}_i\psi\rangle,\quad k\to\infty.$$
In other words,
$$\langle A(b_k)\phi,\psi\rangle\to \langle A(b)\phi,\psi\rangle,\quad k\to\infty.$$
However, by assumption,
$$\langle A(b_k)\phi,\psi\rangle\to \langle A\phi,\psi\rangle,\quad k\to\infty.$$
It follows that
$$\langle A(b)\phi,\psi\rangle=\langle A\phi,\psi\rangle,\quad \phi,\psi\in C^{\infty}_c(\mathcal{G}).$$
Since $C^{\infty}_c(\mathcal{G})$ is dense in $L^p(\mathcal{G})$ and also in $L^{\frac{p}{p-1}}(\mathcal{G}),$ and since $A$ and $A(b)$ are both bounded on $L^p(\mathcal{G}),$ it follows that $A=A(b).$ 

By the preceding paragraph, the graph of the mapping $b\to A(b)$ is closed. Hence, the mapping $b\to A(b)$ is bounded. This completes the proof.
\end{proof}

	\subsection{The lower bound of the norm of the commutator $[M_b, \nabla R_{j}]$}
	In this subsection we prove the lower bound on $[M_b,\nabla R_j],$ that is
	\[
	\|b\|_{\operatorname{Lip}(\mathcal{G})} \lesssim \|[M_b,\nabla R_j]\|_{L^p(\mathcal{G})\to L^p(\mathcal{G})}.
	\]
	
\begin{lemma}\label{lem3.8} 
Let $1\leq i,j\leq n.$ There exist a constant $c_{\mathcal{G}}$ and $g_0\in \mathcal{G}$ (depending on $i$ and $j$) such that $|K^i_j(g^{-1})|\geq c_{\mathcal{G}}$ for every $g\in B(g_0,2).$
\end{lemma}
\begin{proof} The assertion immediately follows from the continuity and homogeneity of the function $K^i_j.$
\end{proof}

\begin{lemma}\label{lem3.9} 
Let $1\leq i,j\leq n.$ There exists a strictly positive constant $c_{\mathcal{G}}$ such that for all  $g \in \mathcal G$ and $r > 0$, we can find $g_*\in \mathcal{G}$ (depending on $i,$ $j,$ $g$ and $r$) such that
$$|K^i_j(g_2^{-1}g_1)|\geq c_{\mathcal{G}}r^{-\mathbb{Q}-1},\quad g_1 \in B(g,r),\quad g_2 \in B(g_*,r).$$
Here, $c_{\mathcal{G}}$ is the constant from Lemma \ref{lem3.8}.
\end{lemma}
\begin{proof} Let $g_0$ be as in Lemma \ref{lem3.8} and set $g_*=g\cdot\delta_r(g_0).$ If $g_1 \in B(g,r),$ then there exists $h_1\in B(1_{\mathcal{G}},1)$ such that $g_1=g\cdot\delta_r(h_1).$ If $g_2 \in B(g_*,r),$ then there exists $h_2\in B(1_{\mathcal{G}},1)$ such that $g_2=g\cdot\delta_r(g_0)\cdot\delta_r(h_2).$ Hence,
$$g_2^{-1}g_1=\delta_r(h_2^{-1}g_0^{-1}h_1).$$
By homogeneity,
$$K^j_j(g_2^{-1}g_1)=r^{-\mathbb{Q}-1}K^i_j((h_1^{-1}g_0h_2)^{-1}).$$
It remains to note that $h_1^{-1}g_0h_2\in B(g_0,2)$ and the assertion follows from Lemma \ref{lem T1 T*1}.
\end{proof}

\begin{lemma}\label{new ji lemma} If $b$ is locally integrable, then
$$c_{\mathcal{G}}\sup_{\substack{g\in\mathcal{G}\\ r>0}}r^{-\mathbb{Q}-1}\int_{B(g,r)}|b(\tilde g)-m_{B(g,r)}b|d\tilde g\leq 8\|[M_b,X_iR_j]\|_{L_p(\mathcal{G})\to L_p(\mathcal{G})},$$
where $c_{\mathcal{G}}$ is a constant in Lemma \ref{lem3.8}.
\end{lemma}
\begin{proof} Fix a ball $B(g,r)$ and find a ball $B(g_*,r)$ as in Lemma \ref{lem3.9}. By continuity, we either have 
$$K_j^i(g_2^{-1}g_1)\geq c_{\mathcal{G}}r^{-\mathbb{Q}-1},\quad g_1\in B(g,r),\quad g_2\in B(g_*,r)$$
or
$$K_j^i(g_2^{-1}g_1)\leq -c_{\mathcal{G}}r^{-\mathbb{Q}-1},\quad g_1\in B(g,r),\quad g_2\in B(g_*,r).$$
For definiteness, we restrict ourselves to the first case.

Let $c$ be the median value of $b$ over $B(g_*,r),$ that is,
$$\left|\left\{g_2\in B(g_*,r):\, b(g_2)<c\right\}\right|,   \left|\left\{g_2\in B(g_*,r):\, b(g_2)>c\right\}\right|\leq\frac12r^{\mathbb{Q}}.$$
Set
$$E_1=\left\{g_1\in B(g,r):\, b(g_1)\leq c\right\},\quad 
E_2=\left\{g_1\in B(g,r):\, b(g_1)>c\right\}$$
and
$$F_1=\left\{g_2\in B(g_*,r):\, b(g_2)\geq c\right\},\quad F_2
=\left\{g_2\in B(g_*,r):\, b(g_2)\leq c\right\}.$$
By construction, $|F_1|,|F_2|\geq \frac12r^{\mathbb{Q}}.$

By H\"older's inequality,
$$|\langle [M_b,X_iR_j]\chi_{F_k},\chi_{E_k}\rangle|\leq \|[M_b,X_iR_j]\|_{L_p(\mathcal{G})\to L_p(\mathcal{G})}r^{\mathbb{Q}},\quad k=1,2.$$
Clearly,
$$\langle [M_b,X_iR_j]\chi_{F_1},\chi_{E_1}\rangle=\int_{E_1\times F_1}(b(g_1)-b(g_2))K_j^i(g_2^{-1}g_1)dg_1dg_2.$$
Since
$$b(g_1)-b(g_2)\leq b(g_1)-c\leq 0,\quad K_j^i(g_2^{-1}g_1)\geq c_{\mathcal{G}}r^{-\mathbb{Q}-1},\quad g_1\in E_1,\quad g_2\in F_1,$$
it follows that
$$|\langle [M_b,X_iR_j]\chi_{F_1},\chi_{E_1}\rangle|\geq c_{\mathcal{G}}r^{-\mathbb{Q}-1}\int_{E_1\times F_1}|b(g_1)-c|dg_1dg_2= \frac1{2r}c_{\mathcal{G}}\int_{E_1}|b(g_1)-c|dg_1.$$
Thus,
$$c_{\mathcal{G}}\int_{E_1}|b(g_1)-c|dg_1\leq 2r^{\mathbb{Q}+1}\|[M_b,X_iR_j]\|_{L_p(\mathcal{G})\to L_p(\mathcal{G})}.$$
Similarly,
$$c_{\mathcal{G}}\int_{E_2}|b(g_1)-c|dg_1\leq 2r^{\mathbb{Q}+1}\|[M_b,X_iR_j]\|_{L_p(\mathcal{G})\to L_p(\mathcal{G})}.$$
Adding the last two inequalities, we obtain
$$c_{\mathcal{G}}\int_{B(g,r)}|b(g_1)-c|dg_1\leq 4r^{\mathbb{Q}+1}\|[M_b,X_iR_j]\|_{L_p(\mathcal{G})\to L_p(\mathcal{G})}.$$
Hence,
$$c_{\mathcal{G}}\int_{B(g,r)}|m_{B(g,r)}b-c|dg_1\leq 4r^{\mathbb{Q}+1}\|[M_b,X_iR_j]\|_{L_p(\mathcal{G})\to L_p(\mathcal{G})}.$$
Adding the last two inequalities, we complete the proof.
\end{proof}

\begin{proof}[Proof of the lower bound in Theorem \ref{Thom1}]
The assertion follows from Lemma \ref{new ji lemma} and Lemma \ref{eq-Rlem1}.
\end{proof}

\section{Proof of Corollary \ref{cor2}}\label{Section co}
\subsection{The upper bound of the norm of the commutator $[M_{b_1},[M_{b_2},\mathrm{X}_{\ell}R_{j}\mathrm{X}_{i}]]$}

\begin{lemma}\label{second cz verification lemma}
If $T$ is a convolution operator with the kernel $K({g'}^{-1}g)$ satisfying 
\begin{align}
&|K(g)| \leq c_K \frac{1}{\rho(g)^{\mathbb Q+2}},\label{eeeee1}\\
&|K(g')-K(g'')| \leq c_K\frac{\rho(g',g'')}{\rho(g')^{\mathbb Q+3}} \quad \text{if} \quad \rho(g')>2\rho(g',g''),\label{eeeee2},
\end{align}
then the kernel of $[M_{b_1},[M_{b_2},T]]$ is a standard kernel as in Definition \ref{def CZO}.
\end{lemma}
\begin{proof} The argument repeats the one in Lemma \ref{first cz verification lemma} and is, therefore, omitted.
\end{proof}

\begin{lemma} The operator $\mathrm{X}_{\ell}R_{j}\mathrm{X}_{i}$ satisfies \eqref{eeeee1}--\eqref{eeeee2}.
\end{lemma}
\begin{proof} The argument is similar to that in Lemma \ref{first czo lemma} and is, therefore, omitted.
\end{proof}

\begin{lemma}
$[M_{b_1},[M_{b_2},\mathrm{X}_{\ell}R_{j}\mathrm{X}_{i}]](1) , [M_{b_1},[M_{b_2},\mathrm{X}_{\ell}R_{j}\mathrm{X}_{i}]]^*(1)  \in \operatorname{BMO}(\mathcal G)$
\end{lemma}
\begin{proof}
We are going to show $[M_{b_1},[M_{b_2},\mathrm{X}_{\ell}R_{j}\mathrm{X}_{i}]](1) , [M_{b_1},[M_{b_2},\mathrm{X}_{\ell}R_{j}\mathrm{X}_{i}]]^*(1)  \in \operatorname{BMO}(\mathcal G)$. Note that
\begin{align*}
	[M_{b_1},[M_{b_2},\mathrm{X}_{\ell}R_{j}\mathrm{X}_{i}]](1)
	&=[M_{b_1},[M_{b_2},\mathrm{X}_{\ell}R_{j}]\mathrm{X}_{i}](1)-[M_{b_1},\mathrm{X}_{\ell}R_{j}\mathrm{X}_{i} M_{\mathrm{X}_{i}b_2} ](1)\\
	&= -[M_{b_2}, X_\ell R_j](\mathrm{X}_ib_1) -[M_{b_1}, X_\ell R_j](\mathrm{X}_ib_2).
\end{align*}
Similarly,
\begin{align*}
	[M_{b_1},[M_{b_2},\mathrm{X}_{\ell}R_{j}\mathrm{X}_{i}]]^*(1)
	&=[M_{b_1},[M_{b_2},\mathrm{X}_{\ell}R_{j}^*]\mathrm{X}_{i}](1)-[M_{b_1},\mathrm{X}_{\ell}R_{j}^*\mathrm{X}_{i} M_{\mathrm{X}_{i}b_2} ](1)\\
	&= -[M_{b_2}, X_\ell R_j^*](\mathrm{X}_ib_1) -[M_{b_1}, X_\ell R_j^*](\mathrm{X}_ib_2).
\end{align*}

Then by noting  that both $[M_{b_1},\mathrm{X}_{\ell}R_{j}]$ and $[M_{b_2},\mathrm{X}_{\ell}R_{j}^*]$ 
are standard Calder\'on--Zygmund operators on $\mathcal G$. Hence, they
are bounded operators from $L^{\infty}(\mathcal G)$ to ${\rm BMO}(\mathcal G)$ (see for example the argument in \cite[Section 6]{FS1982}). 
Then we have  that $[M_{b_1},[M_{b_2},\mathrm{X}_{\ell}R_{j}\mathrm{X}_{i}]] (1) \in \operatorname{BMO}(\mathcal G)$.
Similarly, we conclude that $[M_{b_1},[M_{b_2},\mathrm{X}_{\ell}R_{j}\mathrm{X}_{i}]]^{*} (1) \in \operatorname{BMO}(\mathcal G)$.
The proof is complete.
\end{proof}

\begin{lemma}\label{pre-wbp second lemma} If $b_1,b_2\in {\rm Lip}(\mathcal{G}),$ then
$$\|M_{\chi_{B(g,R)}}[M_{b_1},[M_{b_2},R_j]]M_{\chi_{B(g,R)}}\|_{L^2(\mathcal{G})\to L^2(\mathcal{G})}=O(R^2).$$
\end{lemma}
\begin{proof} The argument repeats the one in Lemma \ref{pre-wbp first lemma} and is, therefore, omitted.
\end{proof}

\begin{lemma}
$[M_{b_1},[M_{b_2},\mathrm{X}_{\ell}R_{j}\mathrm{X}_{i}]]$ satisfies  the Weak Boundedness Property. 
\end{lemma}
\begin{proof} Let $\varphi,\psi\in C^{\infty}_c(\mathcal{G})$ be supported in the ball $B(g_0,R).$ We have
$$[M_{b_2},X_{\ell}R_jX_i]=-M_{X_{\ell}b_2}R_jX_i+X_{\ell}[M_{b_2},R_j]X_i-X_{\ell}R_jM_{X_ib_2}.$$
Thus,
\begin{align*}
[M_{b_1},[M_{b_2},X_{\ell}R_jX_i]]&=-M_{X_{\ell}b_2}[M_{b_1},R_jX_i]+[M_{b_1},X_{\ell}[M_{b_2},R_j]X_i]-[M_{b_1},X_{\ell}R_j]M_{X_ib_2}\\
&=-M_{X_{\ell}b_2}[M_{b_1},R_j]X_i+M_{X_{\ell}b_2}R_jM_{X_ib_1}-M_{X_{\ell}b_1}[M_{b_2},R_j]X_i \\
&\qquad+X_{\ell}[M_{b_1},[M_{b_2},R_j]]X_i-X_{\ell}[M_{b_2},R_j]M_{X_ib_1}
+M_{X_{\ell}b_1}R_jM_{X_ib_2}-X_{\ell}[M_{b_1},R_j]M_{X_ib_2}.
\end{align*}
Hence,
\begin{align*}
\langle [M_{b_1},[M_{b_2},\mathrm{X}_{\ell}R_{j}\mathrm{X}_{i}]]  (\varphi), \psi\rangle&=-\langle [M_{b_1},R_j](X_i\phi),X_{\ell}b_2\cdot\psi\rangle-\langle [M_{b_2},R_j](X_i\phi),X_{\ell}b_1\cdot\psi\rangle\\
&\quad+\langle R_j(X_ib_1\cdot\phi),X_{\ell}b_2\cdot\psi\rangle+\langle R_j(X_ib_2\cdot\phi),X_{\ell}b_1\cdot\psi\rangle\\
&\quad+\langle [M_{b_2},R_j](X_ib_1\cdot\phi),X_{\ell}\psi\rangle+\langle [M_{b_1},R_j](X_ib_2\cdot\phi),X_{\ell}\psi\rangle\\
&\quad-\langle [M_{b_1},[M_{b_2},R_j]](X_i\phi),X_{\ell}\psi\rangle.
\end{align*}
Consequently,
\begin{align*}
|\langle [M_{b_1},[M_{b_2},\mathrm{X}_{\ell}R_{j}\mathrm{X}_{i}]]  (\varphi), \psi\rangle|
&\leq \|M_{\chi_{B(g,R)}}[M_{b_1},R_j]M_{\chi_{B(g,R)}}\|_{L_2(\mathcal{G})\to L_2(\mathcal{G})}\|X_i\phi\|_{L_2(\mathcal{G})}\|X_{\ell}b_2\cdot\psi\|_{L_2(\mathcal{G})}\\
&\quad+\|M_{\chi_{B(g,R)}}[M_{b_2},R_j]M_{\chi_{B(g,R)}}\|_{L_2(\mathcal{G})\to L_2(\mathcal{G})}\|X_i\phi\|_{L_2(\mathcal{G})}\|X_{\ell}b_1\cdot\psi\|_{L_2(\mathcal{G})}\\
&\quad+\|R_j\|_{L_2(\mathcal{G})\to L_2(\mathcal{G})}\|X_ib_1\cdot\phi\|_{L_2(\mathcal{G})}\|X_{\ell}b_2\cdot\psi\|_{L_2(\mathcal{G})}+\|R_j\|_{L_2(\mathcal{G})\to L_2(\mathcal{G})}\|X_ib_2\cdot\phi\|_{L_2(\mathcal{G})}\|X_{\ell}b_1\cdot\psi\|_{L_2(\mathcal{G})}\\
&\quad+\|M_{\chi_{B(g,R)}}[M_{b_2},R_j]M_{\chi_{B(g,R)}}\|_{L_2(\mathcal{G})\to L_2(\mathcal{G})}\|X_ib_1\cdot\phi\|_{L_2(\mathcal{G})}\|X_{\ell}\psi\|_{L_2(\mathcal{G})}\\
&\quad+\|M_{\chi_{B(g,R)}}[M_{b_2},R_j]M_{\chi_{B(g,R)}}\|_{L_2(\mathcal{G})\to L_2(\mathcal{G})}\|X_ib_1\cdot\phi\|_{L_2(\mathcal{G})}\|X_{\ell}\psi\|_{L_2(\mathcal{G})}\\
&\quad+\|M_{\chi_{B(g,R)}}[M_{b_1},[M_{b_2},R_j]]M_{\chi_{B(g,R)}}\|_{L_2(\mathcal{G})\to L_2(\mathcal{G})}\|X_i\phi\|_{L_2(\mathcal{G})}\|X_{\ell}\psi\|_{L_2(\mathcal{G})}.
\end{align*}

The assertion follows now from Lemma \ref{pre-wbp first lemma} and Lemma \ref{pre-wbp second lemma}.
\end{proof}

\subsection{The lower bound of the norm of the commutator  $[M_{b_1},[M_{b_2},\mathrm{X}_{\ell}R_{j}\mathrm{X}_{i}]]$}

\begin{lemma}\label{lem4.6} 
Let $1\leq i,j,\ell\leq n.$ There exists a constant $c_{\mathcal{G}}$ and $g_0\in \mathcal{G}$ (depending on $i$ and $j$) such that $|K^i_{j,\ell}(g^{-1})|\geq c_{\mathcal{G}}$ for every $g\in B(g_0,2).$
\end{lemma}
\begin{proof} The assertion immediately follows from the continuity and homogeneity of the function $K^i_{j,\ell}.$
\end{proof}

\begin{lemma}\label{lem4.7} 
Let $1\leq i,j,\ell\leq n.$ There exists a strictly positive constant $c_{\mathcal{G}}$ such that for all  $g \in \mathcal G$ and $r > 0$, we can find $g_*\in \mathcal{G}$ (depending on $i,$ $j,$ $\ell,$ $g$ and $r$) such that
$$|K^i_{j,l}(g_2^{-1}g_1)|\geq c_{\mathcal{G}}r^{-\mathbb{Q}-2},\quad g_1 \in B(g,r),\quad g_2 \in B(g_*,r).$$
Here, $c_{\mathcal{G}}$ is the constant from Lemma \ref{lem4.6}.
\end{lemma}

\begin{proof}[Proof of the lower estimate in Corollary \ref{cor2}] We fix a ball $B(g,r)$ and consider the ball $B(g_*,r)$ given by Lemma \ref{lem4.7}.

We fix constants $c_1$ and $c_2$ such that
$$|\{h\in B(g_*,r):\ b_1(h)<c_1,\ b_2(h)<c_2\}|,|\{h\in B(g_*,r):\ b_1(h)<c_1,\ b_2(h)>c_2\}|\leq \frac14r^{\mathbb{Q}},$$
$$|\{h\in B(g_*,r):\ b_1(h)>c_1,\ b_2(h)>c_2\}|,|\{h\in B(g_*,r):\ b_1(h)>c_1,\ b_2(h)<c_2\}|\leq \frac14r^{\mathbb{Q}}.$$
The rest of the argument repeats that in Lemma \ref{new ji lemma} and is, therefore, omitted.
\end{proof}

\section{Proof of Theorem \ref{main3}}\label{Scetion T3}

Most of the operators in this section will be on the Hilbert space $L^2(\mathcal{G}),$ and here we will abbreviate ``bounded on $L^2(\mathcal{G})$" as ``bounded". In general, we denote by $\BBB(H)$ the algebra of all bounded linear endomorphisms of a Hilbert space $H.$ The operator norm on $\BBB(H)$ is denoted $\|\cdot\|_{\infty}.$

Note that if the homogeneous dimension $\mathbb Q\leq 2,$ then $\mathcal{G}$ is abelian. We will therefore assume throughout this section that $\mathbb Q>2.$

We will also assume that $\mathcal{G}$ is two-step. This means that the stratification of $\mathfrak{g}$ stops after commutators of order $2,$ that is
\[
\mathfrak{g} = \mathrm{span}\{\mathrm{X}_j,[\mathrm{X}_j,\mathrm{X}_k]\}_{j,k=1}^n
\]
and the commutators $[\mathrm{X}_j,\mathrm{X}_k]$ are central. This includes Heisenberg groups of any dimension.

Recall as in the introduction that given a compact operator $T$ on a Hilbert space $H,$ we denote by $\mu(T) = \{\mu(n,T)\}_{n=0}^\infty$
the singular value sequence of $T.$ We say that $T$ in $\SSS^p(H)$ for $0<p<\infty$ if $\mu(T) \in \ell^p(\mathbb{N}),$ that is if
\[
\|T\|_p := \left(\sum_{n=0}^\infty \mu(n,T)^p\right)^{\frac{1}{p}}<\infty.
\]
Similarly, we say that $T \in \SSS^{p,\infty}(H)$ if $\mu(T) \in \ell^{p,\infty}(\mathbb{N}),$ 
\[
\|T\|_{p,\infty} := \sup_{n\geq 0}\, (n+1)^{\frac1r}\mu(n,T)<\infty.
\]
The Hilbert-Schmidt space is $\SSS^2(H),$ which is a Hilbert space with inner product
\[
\langle T,S\rangle_{\SSS^2} = \mathrm{Tr}(T^*S).
\]

We recall the definition and some properties of double operator integrals. General surveys of double operator integration include \cite{Peller-moi-2016,SkripkaTomskova}. 
We will define double operator integrals on $\BBB(H)$ using the same approach as \cite{PS-crelle} (see also \cite{PSW}). Let $H$ be a complex separable Hilbert space.
Let $D_0$ and $D_1$ be self-adjoint (potentially unbounded) operators on $H$ and let $E^0$ and $E^1$ be the associated spectral measures on $\mathbb{R}.$

For any $A,B\in \SSS^2(H)$, the measure
\begin{align*}
	(\lambda,\mu)\mapsto {\rm Tr}(B^*dE^{0}(\lambda)AdE^{1}(\mu))
\end{align*}
is a countably additive complex-valued measure on $\mathbb{R}^{2}$. For any $\phi\in L^\infty(\mathbb{R}^2)$, $\phi$ is said to be integrable if there is an operator $T_{\phi}^{D_0,D_1}\in \BBB(\SSS^2(H))$ such that for all $A,B\in \SSS^2(H)$,
\begin{align*}
	{\rm Tr}(B^*T_{\phi}^{D_0,D_1}(A))=\int_{\mathbb{R}^{2}}\phi(\lambda,\mu){\rm Tr}(B^*dE^{0}(\lambda)AdE^{1}(\mu)).
\end{align*}
The double operator integral is defined on $A \in \SSS^2(H)$ by
\begin{align}\label{double integral}
	T_{\phi}^{D_0,D_1}(A):=\int_{\mathbb{R}^{2}}\phi(\lambda,\mu)dE^0(\lambda)AdE^1(\mu).
\end{align}

For certain $\phi$, it can be proved that there is a constant $C$ such that
\[
\|T_{\phi}^{D_0,D_1}(A)\|_\infty \leq C\|A\|_{\infty},\quad A \in \SSS^2(H).
\]
In this case, there exists a unique continuous linear extension of $T_{\phi}^{D_0,D_1}$ to $\BBB(H)$ with norm at most $C,$ using the weak density of $\SSS^2(H)$ in $\BBB(H).$ We denote the continuous extension by the same symbol $T_{\phi}^{D_0,D_1}.$

In particular, $T_{\phi}^{D_0,D_1}$ is bounded on $\BBB(H)$ for those functions $\phi$ admitting a decomposition
\[
\phi(\lambda,\mu) = \int_{\Omega} \alpha(\lambda,\omega)\beta(\mu,\omega)\;d\nu(\omega),\quad \lambda,\mu\in \mathbb{R},
\]
where $(\Omega,\Sigma,\nu)$ is a $\sigma$-finite complex measure space, and $\alpha,\beta$ are measurable functions such that
\[
C:= \int_{\Omega} \sup_{\lambda \in \mathbb{R}} |\alpha(\lambda,\omega)|\sup_{\mu\in \mathbb{R}} |\beta(\mu,\omega)|\,d|\nu|(\omega) < \infty.
\]
In this case $T_{\phi}^{D_0,D_1}$ defines a bounded linear operator from $\BBB(H)$ to itself with norm at most $C.$

By an argument involving duality and interpolation, it follows that for every $1\leq p\leq \infty$ we have
\[
\|T_{\phi}^{D_0,D_1}\|_{\SSS^p\to \SSS^p} \leq C.
\]
Moreover, for $1<p<\infty$ we have
\[
\|T_{\phi}^{D_0,D_1}\|_{\SSS^{p,\infty}\to \SSS^{p,\infty}} \leq C.
\]

We record here some basic identities, such as
\begin{equation}\label{doi_cov}
	T^{D_0,D_1}_{\phi(f(\lambda),g(\mu))} = T^{f(D_0),g(D_1)}_{\phi},
\end{equation}
\begin{equation}\label{doi_bimodule}
	T^{D_0,D_1}_{a(\lambda)\phi(\lambda,\mu)b(\mu)}(X) = T^{D_0,D_1}_{\phi}(a(D_0)Xb(D_1))
\end{equation}
and 
\begin{equation}\label{lowner_identity}
	T^{D_0,D_1}_{\frac{f(\lambda)-f(\mu)}{\lambda-\mu}}(D_0X-XD_1) = f(D_0)X-Xf(D_1).
\end{equation}

Of particular importance is the function
\[
\phi(\lambda,\mu) = \frac{2\sqrt{\lambda\mu}}{\lambda+\mu},\quad \lambda,\mu\in \mathbb{R}
\]
with the convention that $\phi(0,0) = 1.$ This function can be decomposed as
\[
\phi(\lambda,\mu) = \int_{-\infty}^{\infty} \lambda^{is}\mu^{-is}\frac{1}{2\cosh(\pi s)}\,ds.
\]
See \cite[Example 3.6]{HiaiKosaki}.

It follows that
\begin{equation}\label{am_gm_inequality}
	\|T^{D_0,D_1}_{\frac{2(\lambda\mu)^{\frac12}}{\lambda+\mu}}\|_{\BBB(H)\to \BBB(H)} \leq \int_{-\infty}^\infty \frac{1}{2\cosh(\pi s)}\,ds = 1.
\end{equation}

The proof of Theorem \ref{main3} is based on some of the results in \cite{MSZ_cwikel},
which we recall here.
\begin{theorem}[Theorems 1.1 and 1.2 in \cite {MSZ_cwikel}]\label{cwikel_estimates}
	Let $\mathcal{G}$ be a stratified Lie group of homogeneous dimension $\mathbb Q$ and with sublaplacian $\Delta,$ as above. 
	
	\begin{enumerate}
		\item{}\label{one_sided_cwikel} If $p>2$ and $f \in L^p(\mathcal{G}),$ then $M_f(-\Delta)^{-\frac{\mathbb Q}{2p}}$ is well defined as a compact operator on $L^2(\mathcal{G}),$ and moreover
		\[
		\|M_f(-\Delta)^{-\frac{\mathbb Q}{2p}}\|_{p,\infty} \leq c_p\|f\|_{L^p(\mathcal{G})}.
		\] 
		\item{} If $p>1$ and $f \in L^p(\mathcal{G}),$ then
		\[
		(-\Delta)^{-\frac{\mathbb Q}{4p}}M_f(-\Delta)^{-\frac{\mathbb Q}{4p}}
		\]
		is well-defined as a compact operator on $L^2(\mathcal{G}),$ and moreover
		\[
		\|(-\Delta)^{-\frac{\mathbb Q}{4p}}M_f(-\Delta)^{-\frac{\mathbb Q}{4p}}\|_{p,\infty} \leq c_p\|f\|_{L^p(\mathcal{G})}.
		\]
	\end{enumerate}
\end{theorem}

The following product estimate follows the same line as classical results.
\begin{lemma}\label{product_lemma}
	If $f \in W^{1,\infty}(\mathcal{G})$ and $u \in W^{\frac12,2}(\mathcal{G}),$ then $uf \in W^{\frac12,2}(\mathcal{G}),$ with
	\[
	\|fu\|_{W^{\frac12,2}(\mathcal{G})} \lesssim \|f\|_{W^{1,\infty}(\mathcal{G})}\|u\|_{W^{\frac12,2}(\mathcal{G})}.
	\]
\end{lemma}
\begin{proof}
	We prove the more general estimate
	\[
	\|fu\|_{\dot{W}^{s,2}(\mathcal{G})} \lesssim (\|f\|_{\dot{W}^{1,\infty}(\mathcal{G})}+\|f\|_{\infty})(\|u\|_{\dot{W}^{s,2}(\mathcal{G})}+\|u\|_{L^2(\mathcal{G})}).
	\]
	for all $0\leq s\leq 1,$ by interpolation from the endpoints $s=0$ and $s=1.$ The $s=0$ case is immediate. For the $s=1$ case, the left hand side is
	\[
	\sum_{k=1}^n \|\mathrm{X}_k(fu)\|_{L^2(\mathcal{G})}
	\]
	which, by the Leibniz rule, is bounded above by
	\begin{align*}
		\sum_{k=1}^n \|\mathrm{X}_k(f)u\|_{L^2(\mathcal{G})}+\|f\mathrm{X}_k(u)\|_{L^2(\mathcal{G})}
		&\leq \sum_{k=1}^n \|f\mathrm{X}_k(u)\|_{L^\infty(\mathcal{G})}\|u\|_{L^2(\mathcal{G})}+\|f\|_{L^\infty(\mathcal{G})}\|\mathrm{X}_ku\|_{L^2(\mathcal{G})}\\
		&= \|f\|_{\dot{W}^{1,\infty}(\mathcal{G})}\|u\|_{L^2(\mathcal{G})}+\|f\|_{L^\infty(\mathcal{G})}\|u\|_{\dot{W}^{1,2}(\mathcal{G})}.
	\end{align*}
	This completes the proof of the  case $s=1$.
\end{proof}

We record the following lemma, the idea of which can be discerned from \cite[Section 5]{MSZ_cwikel}.
\begin{lemma}\label{doi_lemma}
	\begin{enumerate}
		\item{}\label{doi_lemma_1} Let $f \in \dot{W}^{1,\infty}(\mathcal{G})\cap \dot{W}^{2,\mathbb Q}(\mathcal{G}).$
		Then the operator
		\[
		(-\Delta)^{\frac12}[(-\Delta)^{\frac12},M_f](-\Delta)^{-\frac12}
		\]
		is meaningful as a bounded operator on $L^2(\mathcal{G}),$ and moreover
		\[
		\|(-\Delta)^{\frac12}[(-\Delta)^{\frac12},M_{f}](-\Delta)^{-\frac12}\|_{\infty} \lesssim \|f\|_{\dot{W}^{2,\mathbb Q}(\mathcal{G})}+\|f\|_{\dot{W}^{1,\infty}(\mathcal{G})},
		\]
		where the implicit constant is independent of $f.$
		\item{}\label{doi_lemma_2} Similarly, if $f \in \dot{W}^{1,\mathbb Q}(\mathcal{G}),$ then
		\[
		[(-\Delta)^{\frac12},M_f](-\Delta)^{-\frac12}
		\]
		is meaningful as a bounded operator on $L^2(\mathcal{G})$ and belongs to $\SSS^{\mathbb Q,\infty},$ with the norm bound
		\[
		\|[(-\Delta)^{\frac12},M_{f}](-\Delta)^{-\frac12}\|_{\mathbb Q,\infty}\lesssim \|f\|_{\dot{W}^{1,\mathbb Q}(\mathcal{G})},
		\]
		where the implicit constant is independent of $f.$
	\end{enumerate}
\end{lemma}
\begin{proof}
	We have
	\begin{equation}\label{doi_lemma_decomposition}
		(-\Delta)^{\frac12}[(-\Delta)^{\frac12},M_{f}](-\Delta)^{-\frac12} =[-	\Delta,M_{f}](-\Delta)^{-\frac12}-[(-\Delta)^{\frac12},M_{f}].
	\end{equation}
	Writing $\Delta$ as $\sum_{k=1}^n \mathrm{X}_k^2$ and applying the Leibniz rule
	yields
	\begin{align*}
		[-\Delta,M_f](-\Delta)^{-\frac12} &= \sum_{k=1}^n -2[\mathrm{X}_k,M_f]\mathrm{X}_k(-\Delta)^{-\frac12}-[\mathrm{X}_k,[\mathrm{X}_k,M_f]](-\Delta)^{-\frac12}\\
		&= \sum_{k=1}^n -2M_{\mathrm{X}_kf}R_k-M_{\mathrm{X}_k^2f}(-\Delta)^{-\frac12}.
	\end{align*}
	So that
	\[
	\|[-\Delta,M_f](-\Delta)^{-\frac12}\|_{\infty} \leq 2\sum_{k=1}^n \|\mathrm{X}_kf\|_{\infty}\|R_k\|_{L^2(\mathcal{G})\to L^2(\mathcal{G})}+\|M_{\mathrm{X}_k^2f}(-\Delta)^{-\frac12}\|_{\infty}.
	\]
	Bounding the operator norm by the $\SSS^{\mathbb Q,\infty}$ quasinorm and applying (1) in Theorem \ref{cwikel_estimates}  yields
	\[
	\|[-\Delta,M_f](-\Delta)^{-\frac12}\|_{\infty} \leq 2\sum_{k=1}^n \|\mathrm{X}_kf\|_{\infty}\|R_k\|_{L^2(\mathcal{G})\to L^2(\mathcal{G})}+\|M_{\mathrm{X}_k^2f}(-\Delta)^{-\frac12}\|_{\infty} \lesssim \|f\|_{\dot{W}^{1,\infty}(\mathcal{G})}+\|f\|_{\dot{W}^{2,\mathbb Q}(\mathcal{G})}.
	\]
	This proves that the first term on the right hand side of \eqref{doi_lemma_decomposition} is bounded, with the desired norm bound.

	The second term on the right hand side of \eqref{doi_lemma_decomposition} is the commutator $[(-\Delta)^{\frac12},M_f],$ which is bounded on $L^2(\mathcal{G})$ for $f \in \operatorname{Lip}(\mathcal{G}),$ by Corollary \ref{square root of delta corollary}. Since $\|f\|_{\dot{W}^{1,\infty}(\mathcal{G})}\leq \|f\|_{\operatorname{Lip}(\mathcal{G})},$ this completes the proof of \eqref{doi_lemma_1}.

	To prove \eqref{doi_lemma_2}, we instead appeal to \cite[Lemma 4.4]{FLMSZ},
	which implies that
	\[
	\|[(-\Delta)^{\frac12},M_f](-\Delta)^{-\frac12}\|_{\mathbb Q,\infty} \lesssim \|(-\Delta)^{-\frac12}[\Delta,M_f](-\Delta)^{-\frac12}\|_{\mathbb Q,\infty}.
	\]

	We bound the right hand side as follows:
	\begin{align*}
		(-\Delta)^{-\frac12}[\Delta,M_f](-\Delta)^{-\frac12} &= \sum_{k=1}^n (-\Delta)^{-\frac12}\mathrm{X}_k[\mathrm{X}_k,M_f](-\Delta)^{-\frac12}+(-\Delta)^{-\frac12}[\mathrm{X}_k,M_f]\mathrm{X}_k(-\Delta)^{-\frac12}\\
		&= \sum_{k=1}^n R_k^*M_{\mathrm{X}_{k}f}(-\Delta)^{-\frac12}+(-\Delta)^{-\frac12}M_{\mathrm{X}_kf}R_k^*.
	\end{align*}
	The quasi-triangle inequality and Theorem \ref{cwikel_estimates}.\eqref{one_sided_cwikel} yield
	\[
	\|(-\Delta)^{-\frac12}[\Delta,M_{f}](-\Delta)^{-\frac12}\|_{\mathbb Q,\infty} \lesssim \|f\|_{\dot{W}^{1,\mathbb Q}(\mathcal{G})}.
	\]
	This completes the proof of \eqref{doi_lemma_2}.
\end{proof}

One of the main results of \cite{FLMSZ} is that if $f \in \dot{W}^{1,\mathbb Q}(\mathcal{G}),$ then $[R_k,M_f]\in\SSS^{\mathbb Q,\infty},$ with the quasinorm bound
\begin{equation}\label{FLMSZ_sufficiency}
	\|[R_k,M_f]\|_{\mathbb Q,\infty} \lesssim \|f\|_{\dot{W}^{1,\mathbb Q}(\mathcal{G})}.
\end{equation}
This was only proved in \cite{FLMSZ} when $\mathcal{G}$ is the Heisenberg group, however the proof of \eqref{FLMSZ_sufficiency} works equally well in any stratified Lie group. Indeed, writing $R_k$ as $\mathrm{X}_k(-\Delta)^{-\frac12}$ and applying the Leibniz rule yields
\[
[R_k,M_f]=M_{\mathrm{X}_kf}(-\Delta)^{-\frac12}+\mathrm{X}_k[(-\Delta)^{-\frac12},M_f] = M_{\mathrm{X}_kf}(-\Delta)^{-\frac12}-R_k[(-\Delta)^{\frac12},M_f](-\Delta)^{-\frac12}.
\]
A combination of Theorem \ref{cwikel_estimates}.\eqref{one_sided_cwikel} and Lemma \ref{doi_lemma}.\eqref{doi_lemma_2} proves \eqref{FLMSZ_sufficiency}.

The following lemma is a variation on that result. We note that it is precisely at this point we assume that $\mathcal{G}$ is two step.
\begin{lemma}\label{second_order_riesz_commutator_lemma}
	Let $1\leq k_1,k_2\leq n,$ and let $f \in \dot{W}^{1,\mathbb Q}(\mathcal{G}).$
	We have
	\[
	\|[M_f,\mathrm{X}_{k_1}\mathrm{X}_{k_2}(-\Delta)^{-1}]\|_{\mathbb Q,\infty} \lesssim \|f\|_{\dot{W}^{1,\mathbb Q}(\mathcal{G})}.
	\]
\end{lemma}
\begin{proof}
	We decompose the operator $\mathrm{X}_{k_1}\mathrm{X}_{k_2}(-\Delta)^{-1}$ as follows:
	\begin{equation}\label{initial_decomposition_of_double_riesz}
		\mathrm{X}_{k_1}\mathrm{X}_{k_2}(-\Delta)^{-1} = \mathrm{X}_{k_1}(-\Delta)^{-1}\mathrm{X}_{k_2}-\mathrm{X}_{k_1}(-\Delta)^{-1}[-\Delta,\mathrm{X}_{k_2}](-\Delta)^{-1}.
	\end{equation}
	Writing $\Delta=\sum_{k=1}^n \mathrm{X}_k^2,$ the commutator of $\Delta$ with $\mathrm{X}_{k_2}$ is
	\[
	[\Delta,\mathrm{X}_{k_2}] = \sum_{k=1}^n [\mathrm{X}_k^2,\mathrm{X}_{k_2}] = \sum_{k=1}^n \mathrm{X}_k[\mathrm{X}_k,\mathrm{X}_{k_2}]+[\mathrm{X}_k,\mathrm{X}_{k_2}]\mathrm{X}_k.
	\]
	Since $\mathcal{G}$ is two step, the commutator
	\[
	\mathrm{T}_{k,k_2} := [\mathrm{X}_k,\mathrm{X}_{k_2}]
	\]
	is central.
	
	Using this notation, we have
	\[
	[\Delta,\mathrm{X}_{k_2}] = \sum_{k=1}^n 2\mathrm{X}_k\mathrm{T}_{k,k_2}.
	\]
	Applying this to \eqref{initial_decomposition_of_double_riesz} yields
	\begin{align*}
		\mathrm{X}_{k_1}\mathrm{X}_{k_2} &= -R_{k_1}R_{k_2}^{\ast}- 2\sum_{k=1}^n \mathrm{X}_{k_2}(-\Delta)^{-1}\mathrm{X}_{k}\mathrm{T}_{k,k_n}(-\Delta)^{-1}\\
		&= R_{k_1}R_{k_2}^{\ast}- 2R_{k_1}\sum_{k=1}^n R_{k}^{\ast}\cdot \mathrm{T}_{k,k_2}(-\Delta)^{-1}.
	\end{align*}		
	However, since $\mathrm{T}_{k,k_2}$ is central, we have
	$$\mathrm{T}_{k,k_2}(-\Delta)^{-1}=(-\Delta)^{-1/2}\mathrm{T}_{k,k_2}(-\Delta)^{-1/2}=
	(-\Delta)^{-1/2}[\mathrm{X}_k,\mathrm{X}_{k_2}](-\Delta)^{-1/2}
	=-R_k^{\ast}R_{k_2}+R_{k_2}^{\ast}R_k.$$
	
	Thus,
	\[
	\mathrm{X}_{k_1}\mathrm{X}_{k_2}(-\Delta)^{-1} = R_{k_1}R_{k_2}^*-2\sum_{k=1}^n R_{k_1}R_k^*(-R_k^*R_{k_2}+R_{k_2}^*R_k).
	\]		
	It follows from this decomposition, \eqref{FLMSZ_sufficiency} and the Leibniz rule that $[M_f,\mathrm{X}_i\mathrm{X}_j(-\Delta)^{-1}]$ belongs to $\SSS^{\mathbb Q,\infty},$ with quasinorm bounded by $\|f\|_{\dot{W}^{1,\mathbb Q}(\mathcal{G})}.$
\end{proof}

\begin{proof}[Proof of Theorem \ref{main3}]
	Assume initially that $b_0,b_1 \in C^\infty_c(\mathcal{G}).$ We compute the double commutator $[M_{b_1},[M_{b_2},\mathrm{X}_{k_1}R_{k_2}]]$
	as follows. First, the inner commutator is written as
	\[
	[M_{b_2},\mathrm{X}_{k_1}\mathrm{X}_{k_2}(-\Delta)^{-\frac12}] = M_{\mathrm{X}_{k_1}b_2}R_{k_2} + \mathrm{X}_{k_1}M_{\mathrm{X}_{k_2}b_2}(-\Delta)^{-\frac12}+\mathrm{X}_{k_1}R_{k_2}[(-\Delta)^{\frac12},M_{b_2}](-\Delta)^{-\frac12}.
	\]
	Then,
	\begin{align*}
		[M_{b_1},[M_{b_2},\mathrm{X}_{k_1}\mathrm{X}_{k_2}(-\Delta)^{-\frac12}]]&=M_{\mathrm{X}_{k_1}b_2}[M_{b_1},R_{k_2}]\\
		&\quad+M_{\mathrm{X}_{k_1}b_1}M_{\mathrm{X}_{k_2}b_2}(-\Delta)^{-\frac12}\\
		&\quad+\mathrm{X}_{k_1}M_{\mathrm{X}_{k_2}b_2}(-\Delta)^{-\frac12}\cdot [(-\Delta)^{\frac12},M_{b_1}](-\Delta)^{-\frac12}\\
		&\quad+[M_{b_1},\mathrm{X}_{k_1}R_{k_2}](-\Delta)^{-\frac12}\cdot(-\Delta)^{\frac12}[(-\Delta)^{\frac12},M_{b_2}](-\Delta)^{-\frac12}.
	\end{align*}
	Next, by noting that 
	\begin{align*}
		&[M_{b_1},\mathrm{X}_{k_1}R_{k_2}](-\Delta)^{-\frac12}\\
		&\quad=[M_{b_1},\mathrm{X}_{k_1}R_{k_2}(-\Delta)^{-\frac12}]-\mathrm{X}_{k_1}R_{k_2}[M_{b_1},(-\Delta)^{-\frac12}]\\
		&\quad=[M_{b_1},\mathrm{X}_{k_1}R_{k_2}(-\Delta)^{-\frac12}]-\mathrm{X}_{k_1}R_{k_2}(-\Delta)^{-\frac12}\cdot [(-\Delta)^{\frac12},M_{b_1}](-\Delta)^{-\frac12},
	\end{align*}
	we obtain 
	\begin{align*}
		&[M_{b_1},[M_{b_2},\mathrm{X}_{k_1}\mathrm{X}_{k_2}(-\Delta)^{-\frac12}]]\\
		&=M_{\mathrm{X}_{k_1}b_2}[M_{b_1},R_{k_2}]
		+M_{\mathrm{X}_{k_2}b_2}\cdot M_{\mathrm{X}_{k_1}b_1}(-\Delta)^{-\frac12}\\
		&\quad+\mathrm{X}_{k_1}M_{\mathrm{X}_{k_2}b_2}(-\Delta)^{-\frac12}\cdot [(-\Delta)^{\frac12},M_{b_1}](-\Delta)^{-\frac12}\\
		&\quad+[M_{b_1},\mathrm{X}_{k_1}R_{k_2}(-\Delta)^{-\frac12}]\cdot(-\Delta)^{\frac12}[(-\Delta)^{\frac12},M_{b_2}](-\Delta)^{-\frac12}\\
		&\quad-\mathrm{X}_{k_1}R_{k_2}(-\Delta)^{-\frac12}\cdot [(-\Delta)^{\frac12},M_{b_1}](-\Delta)^{-\frac12}\cdot (-\Delta)^{\frac12}[(-\Delta)^{\frac12},M_{b_2}](-\Delta)^{-\frac12}.
	\end{align*}
	
	By the (quasi) triangle inequality for $\SSS^{\mathbb Q,\infty},$ it follows that
	\begin{align*}
		&\|[M_{b_1},[M_{b_2},\mathrm{X}_{k_1}\mathrm{X}_{k_2}(-\Delta)^{-\frac12}]]\|_{\SSS^{\mathbb Q,\infty}}\\
		&\leq2\|M_{\mathrm{X}_{k_1}b_2}\|_{\infty}\|[M_{b_1},R_{k_2}]\|_{\SSS^{\mathbb Q,\infty}}
		+2\|M_{\mathrm{X}_{k_2}b_2}\|_{\infty}\|M_{\mathrm{X}_{k_1}b_1}(-\Delta)^{-\frac12}\|_{\SSS^{\mathbb Q,\infty}}\\
		&\quad
		+2\|\mathrm{X}_{k_1}M_{\mathrm{X}_{k_2}b_2}(-\Delta)^{-\frac12}\|_{\infty}\|[(-\Delta)^{\frac12},M_{b_1}](-\Delta)^{-\frac12}\|_{\SSS^{\mathbb Q,\infty}}\\
		&\quad
		+2\|[M_{b_1},\mathrm{X}_{k_1}\mathrm{X}_{k_2}\Delta^{-1}]\|_{\SSS^{\mathbb Q,\infty}}\|(-\Delta)^{\frac12}[(-\Delta)^{\frac12},M_{b_2}](-\Delta)^{-\frac12}\|_{\infty}\\
		&\quad
		+2\|\mathrm{X}_{k_1}R_{k_2}(-\Delta)^{-\frac12}\|_{\infty}\cdot \|[(-\Delta)^{\frac12},M_{b_1}](-\Delta)^{-\frac12}\|_{\SSS^{\mathbb Q,\infty}}\cdot\| (-\Delta)^{\frac12}[(-\Delta)^{\frac12},M_{b_2}](-\Delta)^{-\frac12}\|_{\infty}.
	\end{align*}

	To continue, we note that, applying \eqref{FLMSZ_sufficiency}, we have
	$$\|M_{\mathrm{X}_{k_1}b_2}\|_{\infty}\|[M_{b_1},R_{k_2}]\|_{\SSS^{\mathbb Q,\infty}}\leq\|b_2\|_{\dot{W}^{1,\infty}}\|b_1\|_{\dot{W}^{1,\mathbb Q}}.$$
	Secondly, by Theorem \ref{cwikel_estimates}.\eqref{one_sided_cwikel}, we have
	$$\|M_{\mathrm{X}_{k_2}b_2}\|_{\infty}\|M_{\mathrm{X}_{k_1}b_1}(-\Delta)^{-\frac12}\|_{\SSS^{\mathbb Q,\infty}}\lesssim\|b_2\|_{\dot{W}^{1,\infty}}\|b_1\|_{\dot{W}^{1,\mathbb Q}},$$
	and
	\begin{align*}
		\|\mathrm{X}_{k_1}M_{\mathrm{X}_{k_2}b_2}(-\Delta)^{-\frac12}\|_{\infty}&\lesssim \|M_{\mathrm{X}_{k_1}\mathrm{X}_{k_2}b_2}(-\Delta)^{-\frac12}\|_{\infty}+\|M_{\mathrm{X}_{k_2}b_2}R_{k_1}\|_{\infty}\\
		&\lesssim\|b_2\|_{\dot{W}^{2,\mathbb Q}(\mathcal{G})}+\|b_2\|_{\dot{W}^{1,\infty}(\mathcal{G})}.
	\end{align*}

	By Lemma \ref{doi_lemma}.\eqref{doi_lemma_2}, we also have
	$$\|[(-\Delta)^{\frac12},M_{b_1}](-\Delta)^{-\frac12}\|_{\SSS^{\mathbb Q,\infty}}\leq \|b_1\|_{\dot{W}^{2,\mathbb Q}(\mathcal{G})}.$$

	In addition, Lemma \ref{doi_lemma}.\eqref{doi_lemma_1} delivers the estimate
	\[
	\|(-\Delta)^{\frac12}[(-\Delta)^{\frac12},M_{b_2}](-\Delta)^{-\frac12}\|_{\infty} \lesssim \|b_2\|_{\dot{W}^{1,\infty}(\mathcal{G})}+\|b_2\|_{\dot{W}^{2,\mathbb Q}(\mathcal{G})}.
	\]
	Lemma \ref{second_order_riesz_commutator_lemma} implies that
	\begin{align}\label{AIM in weak Sd}
		\|[M_{b_1},\mathrm{X}_{k_1}\mathrm{X}_{k_2}\Delta^{-1}]\|_{\SSS^{\mathbb Q,\infty}}\lesssim\|b_1\|_{\dot{W}^{1,\mathbb Q}(\mathcal{G})}. 
	\end{align}
	Combining all the preceding estimates, we arrive at
	$$\|[M_{b_1},[M_{b_2},\mathrm{X}_{k_1}\mathrm{X}_{k_2}(-\Delta)^{-\frac12}]]\|_{\SSS^{\mathbb Q,\infty}}\lesssim\|b_1\|_{\dot{W}^{1,\mathbb Q}(\mathcal{G})}(\|b_2\|_{\dot{W}^{1,\infty}(\mathcal{G})}+\|b_2\|_{\dot{W}^{2,\mathbb Q}(\mathcal{G})}).$$

	Finally, we must remove the assumption that $b_1,b_2 \in C^\infty_c(\mathcal{G}).$ This
	can be achieved by using the density of $C^\infty_c(\mathcal{G})$ in $\dot{W}^{1,\mathbb Q}(\mathcal{G})$, following identical reasoning to the end of the proof of \cite[Theorem 4.5]{FLMSZ}.
\end{proof}

\bigskip\bigskip
\noindent{\bf Acknowledgement:}\ 
Yanping Chen is supported by
National Natural Science Foundation of China (\# 11871096, \# 11471033).
Ji Li is supported by ARC DP 220122085. Dmitriy Zanin is supported by ARC DP 230100434.

\end{document}